\documentclass{amsart}
\usepackage[bbgreekl]{mathbbol}
\usepackage{bbm}
\usepackage{amsmath, amsthm, amssymb, xcolor, enumerate}
\newtheorem{thm}{Theorem}
\newtheorem{conj}[thm]{Conjecture}
\newtheorem{prop}[thm]{Proposition}
\newtheorem{defn}[thm]{Definition}
\newtheorem{lem}[thm]{Lemma}
\newtheorem{fact}[thm]{Fact}
\newtheorem{cor}[thm]{Corollary}
\newtheorem{clm}{Claim}

\renewcommand{\P}{\mathbbm{P}}
\newcommand{\E}{\mathbbm{E}}
\newcommand{\F}{\mathcal{F}}
\newcommand{\G}{\mathcal{G}}
\newcommand{\C}{\mathcal{C}}
\newcommand{\U}{\mathcal{U}}
\newcommand{\CACK}{C_{\ref{ack}}}
\newcommand{\CFKNP}{C_{\ref{fknp}}}
\newcommand{\CMain}{C_{\ref{main}}}
\newcommand{\CConj}{C_{\ref{bestCoarseThreshCol}}}
\newcommand{\kACK}{k_{\ref{ack}}}
\newcommand{\kCH}{k_{\ref{ch}}}
\newcommand{\CSRBM}{C_{\ref{spreadRegBipMatch}}}

\newcommand{\paren}[1]{\left(#1\right)}

\title{Palette Sparsification
via FKNP}
\author{Vikrant Ashvinkumar and Charles Kenney}

\begin{document}

\thanks{We thank Jeff Kahn
for offering comments on an earlier draft.
We thank Jeff Kahn,
Quentin Dubroff,
Natasha Ter-Saakov, and Milan Haiman for
their helpful discussions. Milan 
pointed us toward
Proposition \ref{slackSpread}.}
\begin{abstract}
A random set $S$ is \emph{$p$-spread} if,
for all sets $T$,
$$\P(S \supseteq T) \leq p^{|T|}.$$
There is
a constant $C>1$ large enough that
for every graph $G$ with
maximum degree $D$, there is a $C/D$-spread
distribution on $(D+1)$-colorings of $G$.
Making use of a 
connection between thresholds and spread distributions
due to Frankston, Kahn, Narayanan, and Park
\cite{FKNP}, a
palette sparsification theorem of Assadi,
Chen, and Khanna \cite{ACK} follows.
\end{abstract}
\maketitle
\thispagestyle{empty}
\section{Introduction}\label{Introduction}

In 2018, Assadi, Chen, and Khanna \cite{ACK} proved the following
beautiful and seminal result.

\begin{thm}\label{ack} (\cite{ACK})
Let $G = (V,E)$ be a graph on $n$ vertices
with maximum degree $D$.
There is a constant $\CACK>1$ such that the following holds.
If $\kACK \geq \CACK \log n$ and 
$\forall v \in V,$ $L(v) \in {[D+1] \choose \kACK}$
is chosen uniformly at random and independently, then
$$\P(G \text{ is } L \text{-colorable}) = 1-o_{n \to \infty}(1).$$
\end{thm}

A random set $S$ is $p$\emph{-spread} if, for all sets $T$,
\begin{align}
\label{spreadDef}
\P(S \supseteq T) \leq p^{|T|}.
\end{align}
If $\F$ is a set system (hypergraph) and $\P(S \in \F) = 1$,
we say $\F$ \emph{supports} the random set $S$.
If, moreoveer, $S$ is $p$-spread, we say $\F$
\emph{supports a $p$-spread distribution}. 
We say a hypergraph $\F$ is $\ell$\emph{-bounded}
if $\forall A \in \F$, $|A| \leq \ell$,
and $\ell$\emph{-uniform}
if equality ($|A|=\ell$) holds for all $A \in \F$.
If $\F$ is a hypergraph with groundset $X$, we write
$$\hat{\F} = \{A \in 2^X : \exists B \in \F, B \subseteq A\}.$$
Proving a fractional relaxation
of the Kahn-Kalai conjecture
(conjectured---the relaxation, that is---by Talagrand \cite{Tal}), 
Frankston, Kahn, Narayanan, and Park \cite{FKNP}
gave us the following powerful tool.

\begin{thm}\label{fknp} (\cite{FKNP})
There is a constant $\CFKNP>1$ such that the following holds.
Let $\F$ be an $\ell$-bounded hypergraph on $X$ that supports a $p$-spread distribution.
Let $q = \CFKNP p \log \ell$.
Then
$$\P(X_q \in \hat{\F}) = 1 - o_{\ell \to \infty}(1).$$
\end{thm}
\qed

Theorem \ref{fknp} is a
user-friendly way of proving thresholds.
In 2022, with a short and elegant proof,
Park and Pham \cite{PP} established the full Kahn-Kalai
conjecture. Several authors have since
published their own spins on the proof,
most notably (Bryan) Park and Vondr\'ak \cite{PV},
who extended the theorem to cover non-uniform measures.
All applications of these results known to the present
authors can be
derived from Theorem \ref{fknp} alone.

If $S$ is a random subset of $X$ and $\overset{\to}{p} \in [0,1]^X$,
we generalize (\ref{spreadDef}) and say
$S$ is $\overset{\to}{p}$\emph{-spread}
if for all $A \subseteq X$,
$$\P(S \supseteq A) \leq \prod_{x \in A} p_x.$$
Say $\G$ \emph{covers} $\F$ if
$$\F \subseteq \hat{\G}.$$

\begin{defn}\label{costExpense}
Let $\F$ be a hypergraph on $X$ and $\overset{\to}{q} \in [0,1]^X$.
The $\overset{\to}{q}$\emph{-expense} of $\F$ is
$$e_{\overset{\to}{q}}(\F) = \sum_{A \in \F} \prod_{x \in A} q_x.$$
The $\overset{\to}{q}$\emph{-cost} of $\F$ is
$$c_{\overset{\to}{q}}(\F) = \min_{\F \subseteq \hat{G}} (e_{\overset{\to}{q}}(\G)),$$
where the minimum is taken over all hypergraphs $\G$ that cover $\F$.
\end{defn}

\begin{thm}\label{pv} (\cite{PV})
There is a constant $C>1$ such that the following
holds. Let $\F$ be an $\ell$-bounded hypergraph on
$X$, $\overset{\to}{q} \in (0,1)^X$, 
and $\overset{\to}{p} = C \log(\ell) \cdot \overset{\to}{q}$.
If the $\overset{\to}{q}$-cost of $\F$ is at least 1, then
\begin{align}\label{pvAsympt}
\P(X_{\overset{\to}{p}} \in \hat{\F}) = 1-o_{\ell \to \infty}(1).
\end{align}
In particular, (\ref{pvAsympt}) holds if $\F$ supports
a $\overset{\to}{q}$-spread distribution.\footnote{
Aiming to optimize the constant $C$, Park and Vondr\'ak
state a slightly different result, but the theorem stated here
follows from their proof.
}
\end{thm}
\qed

\noindent
Here $X_{\overset{\to}{p}}$ denotes the random subset
of $X$ given by $\P(x \in X) = p_x$,
independently for all $x \in X$.

We digress to discuss a related area whose methods have
proven useful in this work.
The following conjecture of Casselgren
and H\"aggkvist \cite{CH},
made after 
Johansson \cite{Joh} asked for a threshold, 
has recently been proved
by a flurry of papers:
 
\begin{conj}\label{ch} (\cite{CH})
Let $D$ be an even natural number and
$G = \mathcal{L}(K_{\frac{D}{2}+1, \frac{D}{2}+1})$
the line graph of a complete bipartite graph.
Note that $G$ has $n = (\frac{D}{2}+1)^2$
vertices and maximum degree $D$.

There is a constant $C$ such that
if $\kCH \geq C \log n$, the following holds.
Suppose that for each $v \in V(G)$,
$L(v) \in {[\frac{D}{2}+1] \choose \kCH}$
is chosen uniformly at random, independently
of other choices.
Then
$$\P(G \text{ is } L\text{-colorable}) = 1 - o_{n \to \infty}(1).$$
\end{conj}

Conjecture \ref{ch} differs from
Theorem \ref{ack} in two important ways: first,
there is more structure (we are given the line graph
of a bipartite graph, instead of an arbitrary
graph); and second, the palette of colors allowed,
$[\frac{D}{2}+1]$ instead of $[D+1]$, is much more restrictive.

Andr\'en, Casselgren, and \"Ohman \cite{ACO}
made the first progress toward Conjecture
\ref{ch} in 2012 without the help of Theorem \ref{fknp},
proving that $\kCH = (1 - \Omega(1))\sqrt{n} 
= (\frac{1}{2} - \Omega(1))D$ suffices.

Nearly ten years later in 
April of 2022, Sah, Sawhney, and Simkin \cite{SSS},
using (as would all subsequent work on this problem) 
Theorem \ref{fknp}, reduced $\kCH$ to $n^{o(1)}$.
In June 2022, Kang, Kelly, K\"uhn, Methuku, and Osthus
\cite{KKKMO}
brought $\kCH$ down to $O(\log^2 n)$, within a logarithmic factor
of the optimum.

If $\F$ supports a $p$-spread distribution and $p=O(q)$,
we say \emph{there is an $O(q)$-spread distribution on $\F$}.
Finally, in December 2022, Keevash \cite{Kee}
and Jain and Pham \cite{JP} independently proved:

\begin{thm}\label{keejp} (\cite{Kee} and \cite{JP})
There is an $O(1/D)$-spread distribution on
$[\frac{D}{2}+1]$-colorings of 
$\mathcal{L}(K_{\frac{D}{2}+1, \frac{D}{2}+1})$. Therefore,
by Theorem \ref{fknp},
Conjecture \ref{ch} is true.
\end{thm}
\qed

The main result of the present work is as follows:

\begin{thm}\label{main}
There is a $\CMain > 1$ large enough that the following holds.
Let $G$ be any graph on $n$ vertices with
maximum degree $D$. Then there is a
$\CMain/D$-spread distribution on $[D+1]$-colorings
of $G$. Theorem \ref{ack} follows by Theorem \ref{fknp}.
\end{thm}

Answering a question of Assadi, Kahn and the second
author \cite{KK1} proved an asymptotically optimal
version of Theorem \ref{ack}, reducing $\kACK$ to
$$\kACK = (1+o(1)) \log n.$$
In \cite{KK2} they also prove the following generalization:

\begin{thm}\label{kk2} (\cite{KK2})
Let $G = (V,E)$ be a graph on $n$ vertices with maximum
degree $D$, and suppose that $\forall v \in V,$ we
are given a list $S_v$ of size $D+1$.
There is a $k = (1+o(1)) \log n$ such that
if $\forall v \in V, L(v) \in {S_v \choose k}$
is chosen uniformly at random and independently,
then
$$\P(G \text{ is } L\text{-colorable}) = 1-o_{n \to \infty}(1).$$
\end{thm}
\qed

As far as we know, there is still no way to attain
asymptotically optimal results, such as Theorem \ref{kk2},
from Theorem \ref{fknp}. On the other hand,
our guess would be that any coarse thresholds 
for random-list coloring 
with lists of size $O(\log n)$
\emph{can} be proven by
Theorem \ref{fknp} and its successors.
For example:

\begin{conj}\label{bestSpreadCol}
Let $G = (V,E)$ be a graph on $n$ vertices, and for each
$v \in V$ suppose we are given a list
$S_v$ of size $d_G(v)+1$. Then there is an
$O(1/d_G(v))$-spread distribution on $S$-colorings of $G$.
\end{conj}

Conjecture \ref{bestSpreadCol}, with Theorem \ref{pv}, would imply
the following, which strengthens a result of Halld\'orson,
Kuhn, Nolin, and Tonoyan \cite{HKNT}
and generalizes a theorem of Alon and Assadi
\cite{AA}:

\begin{conj}\label{bestCoarseThreshCol}
(Let $G$, $S$ be as in Conjecture \ref{bestSpreadCol}.)
There is a $\CConj>1$ large enough that the following holds.
If $\forall v \in V, L(v) \in {S_v \choose \CConj \log n}$
is chosen uniformly at random and independently,
then
$$\P(G \text{ is } L\text{-colorable}) = 1-o_{n \to \infty}(1).$$
\end{conj}

The rest of this paper is organized as follows.
In Section \ref{BasicsNotation}, we lay out
notation and definitions, and briefly prove that Theorem \ref{main}
implies Theorem \ref{ack}.
In Section \ref{Tools} we assemble needed tools,
including a `spread Lov\'asz Local Lemma' due to \cite{HaSS}, q.v.
\cite[Prop. 6]{JP}, and a spread bipartite matching theorem
from \cite{PSSS}. Section \ref{Warmups}
proves some elementary spread coloring theorems.
In Section \ref{Warmups} 
we also provide counterexamples to show that the most-natural 
distributions on $(D+1)$-colorings of $G$---the 
uniform and `random-greedy' distributions---may
\emph{not} be sufficiently spread.
After outlining our construction
(i.e. the proof of Theorem \ref{main})
in Section \ref{Outline},
%Section \ref{SpreadRS} contains a spread generalization of a theorem of
%Reed and Sudakov \cite{RS}. 
%\cmnt{This now seems unnecessary.}
Section \ref{Sparse}
constructs a spread coloring of the sparse vertices in $V$.
Also in Section \ref{Sparse}, we propose a conjecture
extending the famous Reed-Sudakov theorem to spread colorings.
Finally, in Section \ref{Clusters} the spread coloring is
extended to the dense clusters in $V$.

\section{Basics and Notation}\label{BasicsNotation}
We write $[n]$ for the set $\{1,2,\cdots,n\}$,
$2^X$ for the powerset of $X$, and 
$${X \choose k} = \{A \in 2^X : |A|=k\}.$$
Define $n^{\underline{k}} = n (n-1) \cdots (n-k+1)$.
In a \emph{list coloring instance}, we are
given a graph $G = (V,E)$ and a function
$L$ on $V$, and our task is to produce
a function $\sigma$ on $V$ such that
\begin{itemize}
\item $\forall v \in V, \sigma(v) \in L(v)$, and
\item $\forall uv \in E, \sigma(u) \neq \sigma(v)$.
\end{itemize}
By tradition, $L(v)$ is called the \emph{list} of $v$.
If such a $\sigma$ exists, we 
call $\sigma$ an $L$\emph{-coloring} and
say $G$ is $L$\emph{-colorable}.
We take $|V|=n$, $d(v) = d_G(v) = |\{u \in V : uv \in E\}|$,
and
$$D = \Delta_G = \max_{v \in V} d(v).$$
(We use `$D$' when referring to the maximum
degree of the graph $G$ from Theorem \ref{main}
and `$\Delta_H$' for other graphs,
$H$, that appear in the argument.)
Our \emph{palette} is $\Gamma = [D+1]$
and we refer to elements $\gamma \in \Gamma$ as
\emph{colors}.

We write $G[S]$ for the subgraph of $G$ induced
on the vertex set $S$; for $T \subseteq V \setminus S$,
$G[S,T]$ is the bipartite subgraph of
$G$ induced on the bipartition $(S,T)$. The number
of edges of $G=(V,E)$ is written as
$e(G) = |E|$.
We write $N_v = N_G(v)$ for the neighborhood of
$v$ in $G$, and $\overline{G} = (V, {V \choose 2} \setminus E)$.

Recall from the introduction that, given $\overset{\to}{p} \in [0,1]^X$, we
write $X_{\overset{\to}{p}}$ for the random subset of $X$
given by $\P(x \in X) = p_x$, independently for all $x$ in $X$.
If $\overset{\to}{p} = (p,p,\cdots,p)$ then we abbreviate
$X_{\overset{\to}{p}} = X_p$.

Given $G=(V,E)$ and lists $S = (S_v : v \in V)$, we
form the \emph{coloring hypergraph} $\F = \F(G,S)$ as follows.
The vertex set of $\F$ is
$$X = \bigcup_{v \in V} \{v\} \times S_v,$$
and for $A \subseteq X,$ we put $A \in \F$
iff $A$ is a proper $S$-coloring of $\F$.
Notice that $\F$ is $n$-uniform (and thus $n$-bounded).

Our use of asymptotic notation $o(\cdot), O(\cdot), \Omega(\cdot)$
is standard, with the implicit limit being as $D \to \infty$.
In contrast, `$0 < a \ll b$' means that `there is a suitable 
function $f$ such that $\forall b>0$, if $a \leq f(b)$,
the following holds'. Larger hierarchies $0 < a \ll \cdots \ll z$
are defined analogously.
In expressions such as $O(\varepsilon D)$, where $0 < \varepsilon \ll 1$,
the implied constant in the $O(\cdot)$ does not depend on $\varepsilon$.

We say a sequence of events $E_m$ holds
\emph{with high probability} (w.h.p.)
if $\lim_{m \to \infty} \P(E_m) = 1$.
The parameter $m$ is usually implicit, but
should be clear from context.

We provide a quick proof that
Theorem \ref{main} together with
Theorem \ref{fknp} implies Theorem \ref{ack}.
We will use the following concentration bound
due to Chernoff \cite{Che}.
\begin{prop}\label{tailConc}
If $\xi$ is binomial or hypergeometric with mean $\E \xi \leq \mu$, and $K>1$,
$$\P(\xi \geq K\mu) < e^{-K \mu \log(K/e)}.$$
\end{prop}
\qed

\begin{proof} \emph{(Theorem \ref{main} \& Theorem \ref{fknp} $\implies$ 
Theorem \ref{ack})}
For every $v \in V$, set $S_v = [D+1]$, and 
let $\F = \F(G,S)$,
i.e. the set of all $(D+1)$-colorings of $G$.
Recall that the ground set of $\F$ is $X = V \times [D+1]$.
Theorem \ref{main} and Theorem \ref{fknp}
imply that, for $q = \CFKNP \CMain \log n / D$ and
$X_q$ the binomial random subset of $X$,
$$\P(\exists \tau, \text{ a proper } D+1 \text{ coloring of } G \text{, s.t. } \tau \subseteq X_q) = 1-o_{n \to \infty}(1).$$
Without worrying too much about the constant, we take $\CACK = 100 \CFKNP \CMain$.
Let $v \in V$ and note that
$$\xi_v := |X_q \cap (\{v\} \times [D+1])|$$
is binomially distributed with mean less than 
%\cmnt{The mean is a touch larger than what's written below, 
%if I'm not making a dumb calculation error. $(D+1)q$ vs $Dq$. 
%I'd add a $2$ factor, and propagate the change (once confirmed).}
$$\mu = 10 \CFKNP \CMain \log n,$$
so by Proposition \ref{tailConc},
$$\P(\xi_v > \CACK \log n) < n^{-10}.$$
Summing over the $n$ possibilities for $v$, we find that,
for
$$E := \{\exists v \text{ s.t. } \xi_v > \CACK \log n\},$$
$$\P(E) < n^{-9}.$$
We now couple the random lists $(L_v : v \in V)$
with $X_q$ as follows. First choose and condition on
$$(\xi_v : v \in V).$$
If the bad event $E$ holds (note this is determined by
the $\xi_v$'s), then we draw $X_q$ independently of the $L_v$'s.
If $\overline{E}$ holds, then we pick $X_q$ by \emph{first}
drawing $L_v$ for all $v \in V$, and \emph{then}
drawing
$$(X_q \cap (\{v\} \times [D+1])) \in {\{v\} \times L_v \choose \xi_v} \text{ u.a.r. and independently.}$$ 
Now,
\begin{align*}
\P(G \text{ is not } L\text{-colorable}) &\leq \P(E) + \P(\not\exists \tau, \text{ a proper } D+1 \text{ coloring of } G \text{, s.t. } \tau \subseteq X_q) \\
&< n^{-9} + o_{n \to \infty}(1) \\
&= o_{n \to \infty}(1).
\end{align*}
\end{proof}

\section{Tools}\label{Tools}

Following \cite{JP}, suppose $(\tau_i : i \in I)$ is a sequence of independent
random variables and $(E_j : j \in J)$ a sequence of events.
We are given sets $S_j \subseteq I$ for each $j \in J$, 
and we assume that $E_j$ is a function of $(\tau_i : i \in S_j)$.
We say $\Lambda$ is a \emph{dependency graph}
on vertex set $J$ if
$$j \not\sim_{\Lambda} j' \ \implies \  S_j \cap S_{j'} = \emptyset.$$

\begin{thm}\label{spreadLLL} (\cite{HaSS}; see \cite[Prop. 6]{JP})
Denote by $\mu$ the usual product measure on
$(\tau_i : i \in I)$ and by $\bbmu$ the conditional measure
$\mu(\cdot | \cap_j \overline{E_j})$. 
Let $p \in (0,1)$, and let $\Lambda$ be a dependency graph
on $J$.
Assume that
$\forall j \in J, \mu(E_j) \leq p$
and
$$4p\Delta_{\Lambda} \leq 1,$$
where $\Delta_\Lambda$ is the maximum degree of $\Lambda$.
Let $E$ be an event determined by a subset $S \subseteq I$
of the variables $\tau_i$. Let $N = |\{j \in J : S \cap S_j \neq \emptyset\}|.$
Then
$$\bbmu(E) \leq \mu(E)\exp(6pN).$$
\end{thm}
\qed

Following \cite{PSSS}, let $B$ be a bipartite graph on $(V_0,V_1)$,
$A_0 \subseteq V_0$, and $A_1 \subseteq V_1$.
The density in $B$ of the pair $(A_0, A_1)$
is
$$\rho_B(A_0, A_1) = \frac{e(B[A_0, A_1])}{|A_0||A_1|}.$$
We say the pair $(V_0, V_1)$ is $\varepsilon$\emph{-regular} 
if for all $A_i \subseteq V_i$ with
$|A_i| \geq \varepsilon |V_i|$ ($i=0,1$),
$$|\rho_B(A_0,A_1) - \rho_B(V_0, V_1)| \leq \varepsilon.$$
We say $B$ is $(\varrho, \delta)$\emph{-super regular}
if:
\begin{itemize}
\item $\rho_B(V_0, V_1) = \varrho$,
\item $(V_0,V_1)$ is $\delta$-regular, and
\item $\forall i \in \{0,1\}, v \in V_i$, 
$d_B(v) \geq (\varrho-\delta)|V_{1-i}|$
\end{itemize}

\begin{thm}\label{spreadRegBipMatch} (\cite{PSSS})
There is a nonincreasing
function $\CSRBM: (0,1] \to (1,\infty)$ 
such that the following holds.
Let $0 < \delta \ll \varrho$. Let $B$ be a $(\varrho, \delta)$-super
regular bipartite graph with parts of size $J$.
There exists a distribution on perfect matchings in $B$
which is $\CSRBM(\varrho)/J$-spread.
\end{thm}
\qed

For the complete, very nice proof of Theorem \ref{spreadRegBipMatch},
see \cite{PSSS}; we briefly digress to discuss the method of proof.
Theorem \ref{spreadRegBipMatch} is proven by first
showing the existence, with $\Omega(1)$ probability,
of a perfect matching in a random $k$-out subgraph $K \subseteq B$.
The subgraph $K$, viewed as a set of edges, 
is itself spread, and the desired spread distribution
on perfect matchings follows from:

\begin{fact}\label{conditionSpreadSubset} (See, e.g. \cite{Kee}, \S 2.2)
Suppose $S$ is $p$-spread and $E$ is an event
with $\P(E) \geq q$. Suppose further that, whenever $E$ holds,
we are given an arbitrary subset $T \subseteq S$.
The distribution on $T$, conditional on $E$, is $p/q$-spread.
\end{fact}
\begin{proof}
Let $R$ be any nonempty set.
\begin{align*}
\P(R \subseteq T | E) &= \P(\{R \subseteq T\} \cap E)/\P(E) \\
&\leq \P(R \subseteq S)/q
\leq p^{|R|}/q
\leq (p/q)^{|R|}.
\end{align*}
\end{proof}

Allowing vertices on both sides of the bipartition
of $B$ to select neighborhoods
is essential: if random neighborhoods $N_K(v) \subseteq N_B(v)$ 
were to be chosen only for $v \in V_0$, 
then $|N_K(v)| = \Omega(\log J)$ would be needed to guarantee
a perfect matching with $\Omega(1)$ probability.
The resulting $K$ would not be $O(1/J)$-spread.
The critical obstruction to the existence of a perfect
matching in a random bipartite graph is isolated vertices
(see, e.g. \cite{JLR} \S 4.1), which are directly
avoided by picking neighborhoods for vertices on both sides.
(For one striking early demonstration of this
phenomenon, see \cite{Wal}.)

In Section \ref{Matching}, we will find ourselves
in a regime not covered
by Theorem \ref{spreadRegBipMatch}, where we need a
spread $V_0$-perfect matching in a bigraph
$B$ on $(V_0, V_1)$ with $|V_0| \leq |V_1|$
and \emph{no universal lower bound on
the $(d_B(v) : v \in V_1)$'s}. 
If, for some $v \in V_1$, $d_B(v) = o(J)$, then a graph $K$ which
selects $O(1)$ random neighbors of $v$
would not be $O(1/J)$-spread. So we
will make use of a different construction to handle
low-degree vertices in $V_1$.
We will just require the following consequence
of Theorem \ref{spreadRegBipMatch}:

\begin{cor}\label{highDegSpreadMatch}
Let $0 < \lambda \ll 1$. Let $F$ be a bipartite
graph with parts $(V_0,V_1)$ of size $I$ and
$\forall v \in V_0 \cup V_1, d_F(v) \geq (1-\lambda)I.$
Then there exists a distribution on perfect
matchings in $F$ which is $O(1/I)$-spread.
\end{cor}
\begin{proof}
Let $\varrho = \rho_F(V_0,V_1) \geq 1-\lambda$
and $\delta = \sqrt{\lambda} \ll \varrho$.
Then $F$ is $(\varrho, \delta)$-super-regular.
Indeed,
\begin{itemize}
\item $\rho_F(V_0, V_1) = \varrho$ by definition;
\item if $A_i \subseteq V_i$ with $|A_i| \geq \delta|V_i|$
($i = 0,1$), then
\begin{align*}
\rho_F(A_0,A_1) &\geq \frac{\delta - \lambda}{\delta} \\
&= 1 - \delta,
\end{align*}
so
\begin{align*}
|\rho_F(A_0,A_1) - \varrho| \leq \delta;
\end{align*}
\item and, $\forall v \in V_0 \cup V_1$,
$d_F(v) \geq (1-\lambda)I > (\varrho - \delta)I.$
\end{itemize}
Therefore, Theorem \ref{spreadRegBipMatch}
implies there is a $\CSRBM(\varrho)/I$-spread
distribution on perfect matchings in $F$. 
In particular, 
since $\CSRBM$ is nonincreasing,
there is a $C/I$-spread distribution
on perfect matchings in $F$, for $C = \CSRBM(\frac{3}{4}),$
say.
\end{proof}

For each $v \in V$,
pick $\tau_v \in \Gamma = [D+1]$ uniformly
at random and independently. Let 
$$T = \{v \in V : \forall w \sim_G v, \tau_w \neq \tau_v\}$$
and, for $v \in V$,
$$P_v = \{uw \in \overline{G}[N_v] : \tau_u = \tau_w \text{ and } \forall z \in (N_v \cup N_u \cup N_w)
\setminus \{u,w\}, \tau_z \neq \tau_u\}.$$
We will need the following fact
which follows from the proof in (\cite{KK1}, \S 6.2):

\begin{thm}\label{martConsequence} (\cite{KK1})
Let $0 < \vartheta \ll 1$, $V^* \subseteq V$,
and suppose that $\forall v \in V^*$,
$e(\overline{G}[N_v]) \geq \vartheta D^2$.
There is a constant $\vartheta' = \Omega(\vartheta)$
such that the following holds. Define the event
$$A_v = \{|N_v \cap T| \neq (e^{-1} \pm \vartheta'/3)D\}
\cup \{|P_v| < \vartheta' D\}.$$
Then,
$$\P(A_v) < e^{-D^{1-o(1)}} \quad \forall v \in V^*.$$ 
\end{thm}
\begin{proof} \textit{(Sketch)} \\
We describe how to extract Theorem \ref{martConsequence} from
\cite[\S 6.2]{KK1}. At the end of ``\textit{Proof of (54)}", on page 16
of \cite{KK1},
it is shown that
$$\P(|X-\E X| > \lambda) = \exp[-\Omega(\lambda^2/(C^2 D))].$$
Here $X$ is a quantity such that 
\begin{align}
\label{EXprops}
\E X \sim e^{-1} D
\end{align}
and 
\begin{align}
\label{Xprops}
|X - |T \cap N_v|| = o(D) \text{ with probability at least }1 - e^{-\omega(D)};
\end{align}
also, $C = D^{o(1)}.$ Setting $\vartheta' = \frac{1}{2} e^{-3} \vartheta$
(q.v. (\ref{reasonForVT'}) below)
and taking
$D$ large enough so that $|\E X - e^{-1} D| < (\vartheta'/9)D$ (using (\ref{EXprops}))
and
$|X - |T \cap N_v|| < (\vartheta'/9)D$
(using (\ref{Xprops}), with failure probability at most $e^{-\omega(D)}$), 
and choosing $\lambda = (\vartheta'/9)D$,
we find
$$\P(||N_v \cap T| - e^{-1} D| > (\vartheta'/3)D) < \exp[-D^{1-o(1)}].$$

In the ``\textit{Proof of (12)}" on Page 17 of \cite{KK1},
$X = |P_v|$ and it is shown that
\begin{align}
\label{reasonForVT'}
\E X > \vartheta e^{-3} D
\end{align}
and
$$\P(|X - \E X| > \lambda) = \exp[-\Omega(\lambda^2/D)].$$
Here, setting $\lambda = \vartheta' D$ gives
$$\P(|P_v| < \vartheta' D) < \exp[-\Omega(\vartheta^2 D)] = \exp[-D^{1-o(1)}],$$
completing the proof by a union bound.
\end{proof}

We require the ingenious sparse-dense decomposition
of Assadi, Chen, and Khanna \cite{ACK}
(expanding on \cite{HSS}, which was
inspired by \cite{Ree}).

\begin{thm}\label{ackSparseDenseDecomp} (\cite{ACK})
Let $0 < \epsilon \ll 1$, and let
$G$ be a $D$-regular graph. There exist
$\vartheta = \Omega(\epsilon^2)$ and
$\varepsilon = O(\epsilon)$, and a partition
$$V = V^* \cup \C_1 \cup \cdots \cup \C_m$$
($m \in \mathbbm{N}$) such that the following holds.

The vertices $v \in V^*$ are \emph{sparse},
in the sense that
$$e(\overline{G}[N_v]) \geq \vartheta D^2.$$

Each set $\C = \C_i$ $(i \in [m])$ is a \emph{cluster}:
\begin{itemize}
\item $\forall v \in \C, |N_v \setminus \C| < \varepsilon D$ and
\item $\forall v \in \C, |\C \setminus N_v| < \varepsilon D.$
\end{itemize}
\end{thm}
\qed

\section{Warmups}\label{Warmups}
\begin{prop}\label{slackSpread}
Let $G=(V,E)$ be a graph and $(S_v : v \in V)$
lists of colors. 
\begin{enumerate}
\item \label{sSLocalDegree} If
$\forall
v \in V, |S_v| \geq (1+\lambda)d(v),$
then there is a $\frac{1}{\lambda d(v)}$-spread
distribution on $S$-colorings of $G$.
\item \label{sSMaxDegree} If
$\forall v \in V, |S_v| \geq (1+\lambda)\Delta_G,$
then there is a $\frac{1}{\lambda \Delta_G}$-spread
distribution on $S$-colorings of $G$.
\end{enumerate}
\end{prop}
\begin{proof}
Order $V = v_1, v_2, \cdots, v_n$ arbitrarily,
and write 
$$N^{-}(v_i) = N_G(v_i) \cap \{v_1, \cdots, v_{i-1}\}.$$
We randomly color $G$ as follows:
for $i = 1, 2, \cdots,$ pick
$\sigma(v_i)$ uniformly at random from
$S_{v_i} \setminus \sigma(N^{-}(v_i))$.
Let $k \in [n]$ and consider a partial
$S$-coloring $\tau =
\{(u_1, \tau_1), \cdots, (u_k, \tau_k)\}.$
When a vertex $u$ picks its color $\sigma(u),$
\begin{align*}
|S_u \setminus \sigma(N^-(u))| &\geq |S_u| - d(u) \\
&\geq \begin{cases}
\lambda d(v) \quad \text{ if we are in (\ref{sSLocalDegree})} \\
\lambda \Delta_G \quad \text{ if we are in (\ref{sSMaxDegree}).}
\end{cases}
\end{align*}
Therefore,
$$\P(\sigma(u_1) = \tau_1, \cdots, \sigma(u_k) = \tau_k)
\leq \begin{cases}
\prod_{i=1}^k \frac{1}{\lambda d(u_i)} \quad \text{ if we are in (\ref{sSLocalDegree})} \\
\paren{\frac{1}{\lambda \Delta_G}}^k \quad \text{ if we are in
(\ref{sSMaxDegree}).}
\end{cases}$$
\end{proof}
Proposition \ref{slackSpread} (\ref{sSLocalDegree})
implies the following (also `warmup') theorem of Alon
and Assadi \cite{AA} (in much
the same way that Theorem \ref{main} implies
Theorem \ref{ack}, but using Theorem \ref{pv}
instead of Theorem \ref{fknp}). Just as in (\cite{AA}, Remark 4.1),
the $d_G(v)$ and $\Delta_G$ of Proposition
\ref{slackSpread} can be replaced with $\kappa_G(v)$
(the degree of $v$ in a degeneracy ordering of $G$)
and $\kappa$ (the degeneracy of $G$), respectively.

\begin{thm}\label{aaSlackRandListCol} (\cite{AA})
Let $G=(V,E)$ be a graph on $n$ vertices, and suppose
$\forall v \in V, |S_v| \geq (1+\lambda) d(v)$.
There is a $k = O(\frac{\log n}{\lambda})$ such that
if each $v \in V$ independently and uniformly
selects $L(v) \in {S_v \choose k}$, then
$G$ is $L$-colorable w.h.p.
\end{thm}
\qed

\subsection{Counterexamples}
Suppose $G=(V,E)$ has maximum degree $D$
and $\forall v \in V, |S_v| \geq D+1$.
Is the \emph{uniform} distribution $\bbsigma$ on
$S$-colorings of $G$, $O(1/D)$-spread? 
Not necessarily. In the following example,
a single vertex-color assignment $\bbsigma(v^*) = \gamma^*$
(the `red thumb') is `too attractive'
to $\bbsigma$, so spread fails.

Let $G = K_{D+1}$ and write $V(G) = \{0,1,\cdots,D\}.$
Let $S_0 = \{0,1,\cdots,D+1\}$ and for $i = 1, \cdots, D,$
$S_i = [D+1]$.
Since $|\{S\text{-colorings } \sigma \text{ s.t. } \sigma(0) = 0\}|$
$=|\{S\text{-colorings } \sigma \text{ s.t. } \sigma(0) \neq 0\}|$
$=(D+1)!$, $P(\bbsigma(0) = 0) = \frac{1}{2} \neq O(1/D).$ 

Even the uniform distribution on $(D+1)$-colorings
(i.e. $S$-colorings with $\forall v \in V, S_v = [D+1]$)
need not be $O(1/D)$ spread.
Assume for simplicity that $D+1$ is a perfect square, and
let $V = U \sqcup W$, where $|U| = \sqrt{D+1}$ and
$|W| = D+1 - \sqrt{D+1}.$
Let $E = E(G) = {V \choose 2} \setminus {U \choose 2}.$
The number $|\{(D+1)\text{-colorings } \sigma
\text{ of } G \text{ s.t. } \forall u \in U,
\sigma(u) = D+1\}| = D^{\underline{D+1 - \sqrt{D+1}}}.$
In total, there are
$$(D+1)^{\underline{D+1 - \sqrt{D+1}}} (D+1)^{\frac{1}{2} \sqrt{D+1}}$$
$(D+1)$-colorings of $G$.
Thus
\begin{align*}
\P(\forall u \in U, \bbsigma(u) = D+1)
&= (D+1)^{-\frac{1}{2}(\sqrt{D+1} + 1)} \\
&= \omega(1/D)^{\sqrt{D+1}}.
\end{align*}

%\cmnt{D-regular example?}
%If including spread R-S, then
%this section should include the
%weaker version that follows from
%Theorem \ref{spreadLLL}.

Another natural distribution on $D+1$-colorings
of $G$ is given by the following simple algorithm:
begin with the empty partial coloring 
$\bbtau = \emptyset$. While the domain
of $\bbtau$ is a proper subset of $V$,
pick a uniformly-random $v \in V \setminus \text{domain}(\bbtau)$
and assign to $v$ a uniformly-random color
from $[D+1] \setminus \bbtau(N_v \cap \text{domain}(\bbtau)).$
We call $\bbtau$ the \emph{random-greedy} distribution,
and for many graphs $G$ it is not $O(1/D)$-spread.
For example, let $G = K_{D,D}$ with bipartition
$([D], [2D] \setminus [D])$. Let $\tau$ be the coloring
$\{(1,1), \cdots, (D,D), (D+1, D+1), \cdots (2D,D+1)\}$
(in words: $\tau$ gives each vertex $i \in [D]$ color $i$,
while vertices $D+1$ through $2D$
all get color $D+1$).
We call $\tau$ the `greedy boys' coloring.
Then
\begin{align*}
\P(\bbtau = \tau) &>
\P(\bbtau \text{ picks vertices } [D] \text{ before } [2D] \setminus [D])
/D! \\
&= \frac{1}{(2D)^{\underline{D}}} > (2D)^{-D} = \omega(1/D)^{2D}.
\end{align*}
 
\section{Outline}\label{Outline}

We now outline our construction of an $O(1/D)$-spread
coloring of $G$. Our approach was inspired by, and
bears a strong resemblance to, the argument in \cite{KK1},
but in every case the argument here is more straightforward.

We build our spread coloring of $G$ in several stages,
implicitly relying on (repeated applications of)
the following facts:

\begin{fact}\label{spreadTogether}
Let $p, q \in [0,1]$. Suppose $S$ is a
$p$-spread random set, and given any possible value
$S_0$ of $S$, $T_{S_0} = (T|\{S=S_0\})$ is a $q$-spread
random set. Then $S \cup T_S$ is $2 \max\{p,q\}$-spread.
\end{fact}
\begin{proof}
Fix $U$.
\begin{align*}
\P[U \subseteq S \cup T_S] 
&= \sum_{V \subseteq U} \P(S \cap U = V)
\sum_{\substack{S_0 \text{ s.t.} \\ S_0 \cap U = V}} \P(S = S_0 | S \cap U = V) 
\P[(U \setminus V) \subseteq T_{S_0}] \\
&\leq \sum_{V \subseteq U} \P(V \subseteq S) \max_{\substack{S_0 \text{ s.t.} \\
S_0 \cap U = V}} (\P[(U \setminus V) \subseteq T_{S_0}]) \\
&\leq \sum_{V \subseteq U} p^{|V|} q^{|U|-|V|} \leq 2^{|U|} (\max \{p,q\})^{|U|}.
\end{align*}
\end{proof}
In particular, if $\sigma$ is an $O(1/D)$-spread
partial coloring of $G$, and 
$\tau = (\tau | \sigma)$ is an $O(1/D)$-spread
extension of $\sigma$,
then $\sigma \cup \tau$ is an $O(1/D)$-spread
partial coloring of $G$.

We may assume $G$ is $D$-regular. Indeed, let
$G'=(V',E')$ be any $D$-regular graph such
that $G = G'[V]$. If $\varphi$ is an $O(1/D)$-spread
distribution on $(D+1)$-colorings of $G'$ then
$\varphi|_{V}$ is an $O(1/D)$-spread distribution
on $(D+1)$-colorings of $G$ (see Fact \ref{conditionSpreadSubset}).

Let $0 < \epsilon \ll 1$. Let $V = V^* \cup \C_1 \cup
\cdots \cup \C_m$ be the sparse-dense decomposition
of Assadi, Chen, and Khanna (Theorem \ref{ackSparseDenseDecomp}). Inspired by
\cite{ACK}, \cite{KK1}, and \cite{KK2}, our construction
of an $O(1/D)$-spread distribution on $(D+1)$-colorings
of $G$ proceeds as follows.

\subsection{Sparse Vertices}
We color $V^*$ in two stages. We always
write $\sigma$ for the expanding partial coloring
of $G$. It will be convenient to work on a regular graph,
so in this phase we keep the vertices of
$\C_1 \cup \cdots \cup \C_m$, but the colors assigned
to those vertices are dummies.

Using Theorem \ref{martConsequence} with 
Theorem \ref{spreadLLL}, we define a spread
partial coloring $\sigma$ on some $T \subseteq V$.
For the remaining
graph $G[V \setminus T]$ and
lists $S_v = \Gamma \setminus \sigma(T \cap N_v)$,
we have (for some small positive $\vartheta'$):
\begin{itemize}
\item for all $v \in V \setminus T$, 
$d_{G[V \setminus T]}(v) = (1 - e^{-1} \pm \vartheta'/3)D$, while
\item for all $v \in V^* \setminus T$,
$$|S_v| \geq (1-e^{-1} + 2\vartheta'/3)D.$$
\end{itemize}
We extend $\sigma$ to 
a spread coloring of
all of $V^* \cup T$ using Proposition \ref{slackSpread}.

\subsection{Clusters}
We now throw out $\sigma|_{V \setminus V^*}$
and spread-color each cluster $\C_i$ assuming a worst-case
labeling $\sigma$ of $V \setminus \C_i$.
Each application of Fact \ref{spreadTogether}
introduces a factor of $2$ to our spread bound,
so we cannot rely on Fact \ref{spreadTogether}
to color $m$-many clusters ($m$ could be $\omega(1)$).
We instead rely on:
\begin{fact}\label{spreadApart}
Let $X$ and $Y$ be disjoint sets.
Suppose $S$ is a $p$-spread subset of $X$,
and given any possible value $S_0$ of $S$,
and $T_{S_0} = (T | \{S = S_0\})$ is a $q$-spread subset of $Y$. 
Then $S \cup T_S$ is a $\max\{p,q\}$-spread subset of $X\cup Y$.
\end{fact}
\begin{proof}
Let $U \subseteq X \cup Y$ and $V = U \cap X.$
Then
\begin{align*}
\P(U \subseteq S \cup T_S) &\leq \P(V \subseteq S) 
\max_{\substack{S_0 \subseteq X \text{ s.t.} \\ V \subseteq S_0}}
\P(U \setminus V \subseteq T_{S_0}) \\
&\leq p^{|V|} q^{|U \setminus V|} \leq (\max\{p,q\})^{|U|}
\end{align*}
\end{proof}
Using Fact \ref{spreadApart}, in order to extend
$\sigma$ to $\C_1 \cup \cdots \cup \C_m$, it suffices
to form an $O(1/D)$-spread coloring $\tau_i$ of each $\C_i$
\emph{given any coloring $\sigma_i$
of $V \setminus \C_i$} (where the implied
constant in the $O(\cdot)$ is uniformly bounded
over the $i \in [m]$).

As in \cite{KK1}, our treatment depends on
the parameter
$$\zeta = e(\overline{G}[\C])/D^2.$$
Let $B$ be the bigraph on $(\C,\Gamma)$
with $v \sim_B \gamma$ iff $\gamma \not\in \sigma(N_v \setminus \C).$
When $\zeta$ is sufficiently small,
we color $\C$ by way of a spread $\C$-perfect
matching in $B$. The matching is constructed in two
steps: first, a small number of vertices are matched
in a greedy random fashion to `unpopular' colors $\gamma \in \Gamma$
(those with $d_B(\gamma)$ too small).
We then show that the leftover vertices
can be spread-matched to remaining colors
by way of Corollary \ref{highDegSpreadMatch}.

For $\zeta$ quite large, it may not be possible
to match $\C$ to $\Gamma$. In this case, we first
extend $\sigma$ to a spread partial coloring
of some subset $U \subseteq \C$, wherein 
\emph{pairs} $u \not\sim_G w$ are assigned 
a common color
$\sigma(u) = \sigma(w)$. 

The leftover
vertices $\C'$ are then spread-matched to
leftover colors $\Gamma'$. Again, the matching is
constructed in two phases, the first of which
matches a subset of $\C'$ to low-$B$-degree colors,
and the second of which appeals to Corollary
\ref{highDegSpreadMatch}.
%\section{Spread Reed-Sudakov}\label{SpreadRS}
%\cmnt{Now seems unnecessary}

\section{Sparse Vertices}\label{Sparse}
For all $v \in V$, draw $\tau_v \in \Gamma$ uniformly at random.
Define $T$ and, for all $v \in V^*$, define $P_v$ and $A_v$ as in
Theorem \ref{martConsequence} and above.
Let $\tilde{\tau}$ be $\tau$
\emph{conditioned on}
$\cap_{v \in V^*} \overline{A_v}$.

The random labeling $\tilde{\tau}$, which
forms a (proper) coloring when restricted to
$T$, is spread by Theorem \ref{spreadLLL}.
For $v \in V^*$, let $W_v = \{w \in V : \text{dist}_G(w,v) \leq 2\}.$
The event $A_v$ is determined
by $(\tau_w : w \in W_v)$, 
so we pick the Lov\'asz graph $\Lambda = G^4[V^*]$.
That is, $\forall v, w \in V^*, v \sim_\Lambda w$ iff $\text{dist}_G(v,w) \in \{1,2,3,4\}$.
We take $p = e^{-D^{1-o(1)}}$ the bound from
Theorem \ref{martConsequence}. Then
$$4 p \Delta_\Lambda \leq 4 e^{-D^{1-o(1)}} D^4 = o(1),$$
so Theorem \ref{spreadLLL} applies to $\tilde{\tau}$.

%\cmnt{Perhaps change this to $V^* \times \Gamma$ since that's what we care about?}
Let $(v_1, \gamma_1), \cdots, (v_k, \gamma_k) \in V^* \times \Gamma$.
Let $S = \{v_1, \cdots, v_k\}$ and let $N$ be the number
of $v \in V^*$ such that $S \cap W_v \neq \emptyset.$
Then $N \leq kD^2$. Therefore
\begin{align}
\label{spreadLLLApp}\P(\forall i \in [k], \tilde{\tau}(v_i) = \gamma_i)
&\leq \P(\forall i \in [k], \tau(v_i) = \gamma_i) \exp(6pN) \\
\nonumber &\leq (D+1)^{-k} \exp(k/e^{D-o(D)}) \\
\nonumber &= O(1/D)^k,
\end{align}
where (\ref{spreadLLLApp}) is by Theorem \ref{spreadLLL}.

We let $\sigma|_T = \tilde{\tau}|_T$ and proceed to color $V^* \setminus T$.
Let $G' = G[V^* \setminus T]$
and $\forall v \in V^* \setminus T,
S_v = \Gamma \setminus \sigma(T \cap N_v)$. 
For any $v \in V^* \setminus T$, because $A_v$ holds,
$$d_{G'}(v) \leq |N_v \setminus T| \leq (1 - e^{-1} + \vartheta'/3)D.$$
Meanwhile,
\begin{align*}
|S_v| &\geq D+1 - |N_v \cap T| + |P_v| \\
&\geq (1 - e^{-1} + 2\vartheta'/3)D.
\end{align*}
We apply Proposition \ref{slackSpread} (\ref{sSMaxDegree})
to $G'$ with
$$\Delta_{G'} = (1 - e^{-1} \pm \vartheta'/3)D, \ \ 
\lambda > \vartheta'/3$$
to get a $\frac{1}{\lambda \Delta_{G'}} = O(\frac{1}{\vartheta D})$-spread
distribution 
$\sigma|_{V^* \setminus T}$
on $S$-colorings of $G'$.
The union
$$\left(\sigma|_T\right) \cup \left(\sigma|_{V^* \setminus T}\right)$$
is $O(1/\vartheta D)$-spread by Fact \ref{spreadTogether}.
This completes the construction of the spread
coloring $\sigma$ of $V^*$.
As promised, we now drop $\sigma|_{T \setminus V^*}.$

\subsection{Spread Reed-Sudakov}
An earlier version of our argument for sparse
vertices would have relied on the following,
which, though no longer needed here, may
be of independent interest.
\begin{conj}\label{sparseRS}
There is a universal constant $C$
such that for any $\varrho>0$
and $t \geq t_{\varrho}$,
the following holds. Let $G$ be a graph
with lists $(S_v : v \in V(G))$,
and define $d_\gamma(v) := |\{w \in N_v : \gamma \in S_w\}|.$
If
\begin{itemize}
\item $|S_v| \geq (1 + \varrho)t$ $\ (\forall v)$ and
\item $d_\gamma(v) \leq t$ $\ (\forall v, \gamma \in S_v)$,
\end{itemize}
then there is a $C/t$-spread distribution on $S$-colorings of $G$.
\end{conj}

\section{Clusters}\label{Clusters}
Let $\C$ be a cluster in $G$: that is,
for some $\varepsilon \ll 1,$ we have,
$\forall v \in \C$,
\begin{align}\label{clusterLowOutDegree}
|N_v \setminus \C| < \varepsilon D
\end{align}
and
\begin{align}\label{clusterLowNonDegree}
|\C \setminus N_v| < \varepsilon D.
\end{align}
For any function $\sigma|_{V \setminus \C}$
taking values in $\Gamma$, we construct a
spread proper $(D+1)$-coloring $\sigma|_\C$ 
such that $\forall v \in \C, \forall w \sim v,
\sigma(v) \neq \sigma(w).$

Let $H = \overline{G}[\C]$, and define
$\zeta = e(H)/D^2$. Let
$$\varepsilon \ll \zeta_0 D \ll 1,$$
and say $\zeta$ is \emph{large} if
$\zeta \geq \zeta_0$ (and \emph{small}
otherwise). 
For small $\zeta$, we go directly to
\S \ref{Matching}.
When $\zeta$ is large, we begin
with the following process assigning
colors to pairs in $H$.

\subsection{Process}\label{Process}
Let $B$ be the bigraph on $(\C, \Gamma)$
with $v \sim_B \gamma$ iff $\gamma \not\in \sigma(N_v)$,
i.e. if $\gamma$ is a legal color at $v$.
We now form a spread \emph{partial} coloring $\pi$ of $\C$
such that $\forall \gamma \in \text{range}(\pi),$
$|\pi^{-1}(\gamma)| = 2$.

Pick $\eta$ such that
\begin{align}\label{defEta}
\frac{1}{D}, \zeta \ll \eta \ll \zeta/\varepsilon, 1.
\end{align}
Having honed in on one cluster $\C$, we recycle the
notation $\C_i$ for this process: let $\C_0 = \C,
H_0 = H,$ and $\Gamma_0 = \Gamma$.
For $i = 1,2, \cdots, \eta D$,
\begin{enumerate}[(I)]
\item \label{pickEdge} Pick $u_i v_i  \in E(H_{i-1})$ uniformly at random.
\item \label{pickColor} Pick $\gamma_i \in N_B(u_i) \cap N_B(v_i) \cap \Gamma_{i-1}$
uniformly at random.
\item \label{update} Set $\pi(u_i) = \pi(v_i) = \gamma_i$.
Let $\C_i = \C_{i-1} \setminus \{u_i,v_i\},$
$H_i = H[\C_i]$, and $\Gamma_i = \Gamma_{i-1} \setminus \{\gamma_i\}.$ 
\end{enumerate}

For any $i \in [\eta D]$,
\begin{align}\label{processManyEdges}
e(H_{i-1}) > (\zeta - 2 \eta \varepsilon)D^2
\end{align}
and for any $uv \in E(H_{i-1})$,
\begin{align}\label{processManyColors}
|N_B(u) \cap N_B(v) \cap \Gamma_{i-1}| > (1 - 2\varepsilon - \eta)D.
\end{align}
For (\ref{processManyEdges}):
fewer than $2\eta D$ vertices have been
removed in the first $i-1 < \eta D$ rounds of the process,
and each such vertex has $H$-degree 
less than $\varepsilon D$ by (\ref{clusterLowNonDegree}).
For (\ref{processManyColors}):
$|\Gamma \setminus N_B(v)| < \varepsilon D$
by (\ref{clusterLowOutDegree}), and at-most $i-1 < \eta D$
colors have been removed in the preceding rounds
of the Process.

Let $k, l \in \mathbb{N}$ such that $k+2l \leq \eta D$
and $\{(w_1, \phi_1), \cdots, (w_k, \phi_k),$
$(x_1, \phi_{k+1}), (y_1, \phi_{k+1}),$
$\cdots (x_l, \phi_{k+l}),$
$(y_l, \phi_{k+l})\} = \tau \subseteq \C \times \Gamma$. Assume
$w_1, \cdots, w_k, x_1, \cdots, x_l, y_1, \cdots, y_l$
are pairwise distinct, as are $\phi_1, \cdots \phi_{k+l}$.

By (\ref{processManyColors}), 
for any $j \in [k]$ and
given \emph{any}
history $\mathcal{H}$ in the rounds preceding
the selection of $w_j$,
$$\P(\pi(w_j) = \phi_j | \mathcal{H}) \leq \frac{1}{(1 - 2 \varepsilon - \eta)D}.$$
By (\ref{processManyEdges}) and (\ref{processManyColors}),
for any $j \in [l]$,
\begin{align}
\nonumber \P(\pi(x_j) = \pi(y_j) = \phi_{k+j}) 
&\leq \P(\exists i \in [\eta D] \text{ s.t. } u_i v_i = x_j y_j)/
((1 - 2\varepsilon - \eta)D) \\
\nonumber &\leq \frac{\eta D}{(\zeta - 2 \eta \varepsilon) 
D^2 (1 - 2 \varepsilon - \eta) D}
\ll \frac{1}{\varepsilon D^2}.
\end{align}
Therefore $\pi$ is spread:
\begin{align*}
\P(\pi \supseteq \tau) &\leq \paren{\frac{1}{(1-2\varepsilon -\eta)D}}^k
\paren{\frac{1}{\varepsilon D^2} }^l \\
&\leq (1/\sqrt{\varepsilon}D)^{k+2\ell}.
\end{align*} 

We now let $\C' = \C \setminus \text{domain}(\pi)$ and
$\Gamma' = \Gamma \setminus \text{range}(\pi)$.

\subsection{Matching}\label{Matching}

\begin{lem}\label{spreadMatch}
Let $0 < z \ll 1$.
Let $B$ be a bigraph on $(X,Y)$ with
$|X|=J$, $|Y| = J+R$, with $0 \leq R \leq zJ$.
Suppose that
\begin{align*}
\forall x \in X, d_B(x) \geq J - r_x
\end{align*}
and
\begin{align*}
\max_{x \in X} (r_x) \leq zJ, \quad \sum_{x \in X}r_x \leq zJ
\end{align*}
($r_x \leq 0$ is allowed).
There is an $O_{J \to \infty}(\frac{1}{\sqrt{z} J})$-spread
$X$-perfect matching in $B$.
\end{lem}
\begin{proof}
Since $z$ is fixed and $J \to \infty$,
we assume wherever necessary that $J$ is
large enough to support our claims.
Let $\delta = 5 \sqrt{z} \ll 1$ and define the `unpopular' colors 
\begin{align*}
\U = \{y \in Y : d_B(y) < (1-\delta)J\}. 
\end{align*}
We have
\begin{align*}
|\U|(1-\delta)J + (J+R-|\U|)J &\geq \sum_{x \in X} d_B(x) \geq J^2 - zJ,
\end{align*}
so
\begin{align}
\label{Usmall} |\U| \leq R/\delta + \delta < \frac{\delta J}{20}.
\end{align}
Let 
\begin{align}\label{defr}
|\U| = R + r
\end{align}
If $r \leq 0$, we go straight to
\emph{Matching with High-Degree Colors}. Until then,
assume $r>0$.

We have
\begin{align*}
e(B[X,Y \setminus \U]) \leq J(J-r),
\end{align*}
so
\begin{align}
\nonumber e(B[X,\U]) &\geq \paren{\sum_{x \in X} d_B(x) } - J(J-r) \\
\label{totalUnpopEdges} &\geq J(r-z).
\end{align}
%\cmnt{What do you think of a comment addressing $r < 0$? Maybe a parenthetical (skipping if $r < 0$)}
We first form a spread matching
of size $r$
in $B[X,\U]$ by the following
simple algorithm.
Let $X_0 = X, \U_0 = \U$, $M_0 = \emptyset$ and for $i = 1,2, \cdots, r$,
\begin{enumerate}[(I)]
\item Pick $x_i y_i \in B[X_{i-1}, \U_{i-1}]$ uniformly at random
and add $x_i y_i$ to the matching: $M_i = M_{i-1} \cup \{x_i y_i\}$.
\item Let $X_i = X_{i-1} \setminus \{x_i\}$, $\U_i = \U_{i-1} \setminus \{y_i\}$.
\end{enumerate}
Let $M = M_r$.

\begin{clm}\label{loseFew}
For all $i \in [r],$ $e(B[X_i, \U_i]) \geq e(B[X_{i-1}, \U_{i-1}]) 
- |\U| - (1-\delta)J$
\end{clm}
\begin{proof}
\renewcommand{\qedsymbol}{$\square_1$}
The vertex $x_i$ has at most $|\U|$ neighbors in $\U_{i-1}$
and $y_i$ has at most $(1-\delta)J$ neighbors in $X_{i-1}$.
Removing all these edges, the claim follows.
\end{proof}

It follows that
for all $i \in \{0,1,\cdots,r\}$,
\begin{align}
\nonumber e(B[X_i, \U_i]) &\geq e(B[X_r, \U_r]) \\
\label{totalUnpopEdgesApp} &\geq J(r-z) - r(|\U| + (1-\delta)J) \\
\label{UsmallApp} &\geq J \left[
(r-z) - r\paren{1-\frac{2\delta}{3}}
\right] \geq \frac{r\delta J}{2}.
\end{align}
Here (\ref{totalUnpopEdgesApp}) is by (\ref{totalUnpopEdges})
and Claim \ref{loseFew}. Inequality (\ref{UsmallApp})
follows from (\ref{Usmall}):
$$(1-\delta)J + |\U| \leq (1 - \delta) J + \frac{\delta J}{20} < (1-2\delta/3)J.$$

\begin{clm}
$M$ is $(\frac{2}{\delta J})$-spread.
\end{clm}
\begin{proof}
\renewcommand{\qedsymbol}{$\square_2$}
Let $k \leq r$.
Let $u_1, \cdots, u_k \in X$, and $v_1, \cdots, v_k \in \U$.
Then
\begin{align}
\nonumber
\P\left[\{u_1 v_1, \cdots, u_k v_k\} \subseteq M \right] 
& \leq 
\sum_{\overset{\to}{i} \in [r]^k} \P\left[\forall j \in [k], u_j v_j \text{ picked in round } i_j\right]
\\
\label{loseFewApp}
& \leq
r^k \paren{ \frac{2}{\delta r J} }^k
=
\paren{\frac{2}{\delta J}}^k
,
\end{align}
showing that $M$ is $\frac{2}{\delta J} = O(\frac{1}{\sqrt{z}J})$-spread.
The first inequality in (\ref{loseFewApp}) follows from
(\ref{UsmallApp}).
\end{proof}

\textbf{Matching with High-Degree Colors.}
If $r > 0$, the vertices of $X_r$,
and if $r \leq 0$, \emph{all of} $X$, is yet to be matched.  
(See (\ref{defr}) for $r$.)
We now rectify this by
applying Corollary \ref{highDegSpreadMatch} in
what's left of $B$ (after removing $M$
and $\U$ if $r>0$, or after removing the $R$ least-popular
colors if $r \leq 0$).
Order $Y = y_1, y_2, \cdots, y_{J+R}$, where
$d_B(y_1) \geq \cdots \geq d_B(y_{J+R}).$ Let
\begin{align*}
V_0 = \begin{cases}
X_r \quad \text{ if } r>0 \\
X \quad \text{ if } r \leq 0,
\end{cases}
\end{align*}
and
\begin{align*}
V_1 = \begin{cases} 
Y \setminus \U \quad \text{ if } r>0 \\
\{y_1, \cdots, y_J\} \text{ if } r \leq 0.
\end{cases}
\end{align*}
Define
$$I = |V_0| = |V_1| = J-\max\{r,0\}.$$
We take $F = B[V_0, V_1]$.
Set $\lambda = 2\delta \ll 1$. For every $v \in V_0 \cup V_1,$
$$d_F(v) \geq (1-\lambda)I.$$
Indeed, for $x \in V_0$,
\begin{align*}
d_F(x) &\geq J - r_x - \max\{|\U|,R\} \\
&\geq (1-z-\delta/20)J \geq (1-\lambda)I.
\end{align*}
For $y \in V_1,$
\begin{align*}
d_F(y) &\geq (1-\delta)J - \max\{r,0\} \\
&\geq (1-\delta)J - |\U| \geq (1 - \lambda)I.
\end{align*}
Corollary \ref{highDegSpreadMatch} now finds an
$O(1/I) < O(\frac{1}{\sqrt{z}J})$-spread perfect matching
$M'$ in $F$. The desired $O(\frac{1}{\sqrt{z}J})$-spread 
$X$-perfect matching in $B$
is $M \cup M'$ if $r>0$, and just $M'$ if $r\leq 0$.
($M \cup M'$ is $O(\frac{1}{\sqrt{z}J})$-spread
as well, by one
application of Fact \ref{spreadTogether} or \ref{spreadApart}.)
%\cmnt{Maybe explain why we can union the two
%and still get something spread}
\end{proof}

\textbf{Small $\zeta$.}
We apply Lemma \ref{spreadMatch}
with $X = \C, Y = \Gamma,$ and
$$v \sim_B \gamma \iff \gamma \in \Gamma \setminus \sigma(N_v).$$
So $J = |\C|$, $R = D+1 - |\C|$, and
for each $x \in \C,$
$$d_B(x) \geq J - d_H(x)$$
(recall $H = \overline{G}[\C]$), so we take $r_x = d_H(x)$.
Let $z = 3(\varepsilon + \zeta D).$
We need to show $R, \max_{x \in \C} r_x$, and $\sum_{x \in \C} r_x$
are all at most $zJ$.
First,
\begin{align*}
R \leq \varepsilon D < zJ
\end{align*}
by (\ref{clusterLowOutDegree}); by (\ref{clusterLowNonDegree}),
\begin{align*}
\max_{x \in \C} r_x \leq \varepsilon D < zJ;
\end{align*}
and $\sum_{x \in X} r_x = 2\zeta D^2 < zJ.$
We have $z \ll 1$ since $\varepsilon \ll 1$ and 
(because $\zeta$ is small)
$\zeta \leq \zeta_0 \ll 1/D$.
The last hypothesis of Lemma \ref{spreadMatch}
we have yet to show is $0 \leq R$:
if $|\C| \geq D+2,$ then $e(H) \geq \frac{D}{2},$
i.e. $\zeta \geq \frac{1}{2D}.$ But $\zeta \leq \zeta_0 \ll \frac{1}{D}$.

So, we can color the clusters $\C$ for which $\zeta$ is small
using our $O_{J \to \infty} (\frac{1}{\sqrt{z}J}) =$ 
$O_{D \to \infty} (\frac{1}{\sqrt{\varepsilon}D})$-spread 
\emph{matching} from $\C$ into $\Gamma$
(i.e., if $v$ is matched to $\gamma$ in the matching,
then we color $v$ with $\gamma$).

\textbf{Large $\zeta$.} We apply Lemma \ref{spreadMatch}
with $X = \C'$, $Y = \Gamma'$, and
$$v \sim_B \gamma \iff \gamma \in \Gamma' \setminus \sigma(N_v)$$ 
(see the end of \S \ref{Process} for $\C'$, $\Gamma'$).
So $J = |\C'| = |\C| - 2\eta D$ and $R = D+1 - |\C| + \eta D$.
Since $G$ is $D$-regular
(see the discussion at the beginning of \S \ref{Outline}), for any $x \in \C'$,
$x$ has $(d_H(x) + D+1 - |\C|)$
neighbors in $V \setminus \C$. Thus
\begin{align*}
d_B(x) &\geq D+1 - \eta D - (d_H(x) + D+1 - |\C|) \\
&= J + \eta D - d_H(x);
\end{align*}
so set $r_x = d_H(x) - \eta D$.
Let $z = 2(\varepsilon + \eta).$ We again need to show
$R, \max_{x \in \C'} r_x,$ and $\sum_{x \in \C'} r_x$
are all at most $zJ$. First,
\begin{align}
\nonumber R &= D+1 - |\C| + \eta D \\
\label{smallR} &< \varepsilon D + \eta D < 2(\varepsilon + \eta) J,
\end{align}
where the first inequality in (\ref{smallR}) follows from (\ref{clusterLowOutDegree})
and the fact that $G$ is $D$-regular.
Next, $\max_{x \in \C'} r_x \leq \max_{x \in \C} d_H(x) \leq \varepsilon D < zJ$
by (\ref{clusterLowNonDegree}); and
$\sum_{x \in \C'} r_x = 2 \zeta D^2 - \eta D J < 0 < zJ$
since $\eta \gg \zeta$ and $J = \Omega(D)$.
Now we show $R \geq 0$. If $|\C| \leq D+1$, this is immediate.
Otherwise, since $|\C| - D - 1 \leq d_H(x)$ for every $x \in \C,$
$2\zeta D^2 = \sum_{x \in \C} d_H(x) \geq |\C|(|\C| - D - 1)$,
so $|\C| - D - 1 < 2 \zeta D.$ Hence $R > \eta D - 2 \zeta D > 0.$
(Recall from (\ref{defEta}) that $\eta \gg \zeta$.)

We thus color the clusters $\C$ with large $\zeta$ as follows.
First, we use the Process in \S \ref{Process} to give an
$O(\frac{1}{\sqrt{\varepsilon D}})$-spread 
coloring $\pi$ to $\C \setminus \C'$.
Then, Lemma \ref{spreadMatch} gives us an $O_{J \to \infty} (\frac{1}{\sqrt{z}J}) =$
$O_{D \to \infty} (\frac{1}{\sqrt{\varepsilon}D})$-spread
\emph{matching} from $\C'$ to $\Gamma'$, with which we color the leftover vertices $\C'$
(assigning color $\gamma$ to $v$ if $v$ is matched to $\gamma$).
\qed

\end{document}